\documentclass[a4paper,10pt]{amsart}

\usepackage[cp1250]{inputenc}
\usepackage{graphicx}
\usepackage{amsmath,amsthm}
\usepackage{amssymb}
\usepackage{amscd}
\usepackage{eucal}

\newcommand{\BB}{\mathbb}

\newcommand{\on}{\operatorname}
\newcommand{\Var}{\on{Var}}

\newcommand{\mc}{\mathcal}

\newcommand{\ce}{\mathfrak c}

\newcommand{\ve}{\varepsilon}

\newtheorem{thm}{Theorem}
\newtheorem{fact}[thm]{Fact}
\newtheorem{lem}[thm]{Lemma}
\newtheorem{prop}[thm]{Proposition}

\theoremstyle{definition}

\theoremstyle{remark}

\title[Nonseparable spaceability and strong algebrability]{Nonseparable spaceability and strong algebrability of sets of continuous singular functions}

\author{Marek Balcerzak}
\address{Institute of Mathematics,
         \L\'od\'z University of Technology, W\'olcza\'nska 215,
         93-005 \L\'od\'z,
         Poland}
\email{marek.balcerzak@p.lodz.pl}
\author{Artur Bartoszewicz}
\address{Institute of Mathematics,
         \L\'od\'z University of Technology, W\'olcza\'nska 215,
         93-005 \L\'od\'z,
         Poland}
\email{artur.bartoszewicz@p.lodz.pl}
\author{Ma{\l}gorzata Filipczak}
\address{Faculty of Mathematics and Computer Science,
         University of \L\'od\'z, Banacha 22,
         93-008 \l\'od\'z
         Poland}
\email{malfil@math.uni.lodz.pl}

\date{}

\begin{document}
\begin{abstract}
Let $\on{CBV}$ denote the Banach algebra of all continuous real-valued functions of bounded variation, defined in $[0,1]$.
We show that the set of strongly singular functions in $\on{CBV}$
is nonseparably spaceable. We also prove that certain families of singular functions constitute strongly \mbox{$\ce$-algebrable}
sets. The argument is based on a new general criterion of strong \mbox{$\ce$-algebrability}.
\end{abstract}

\maketitle
\section{Introduction}
In the last decade, much work was done in the study of subsets of vector spaces (topological vector spaces, normed spaces, Banach algebras, etc.)
with no linear structure given a priori. This research was earlier initiated by Gurariy \cite{G1}, \cite{G2} and then continued by several authors.
See for instance \cite{AGS}, \cite{ASS}, \cite{BG1}, \cite{BG2}, \cite{BG}, \cite{GKP}, \cite{JMS}.

Recall (see \cite{AGS}) that, for a topological vector space $V$, its subset $A$ is said to be:
\begin{itemize}
\item {\em lineable} if $A\cup\{ 0\}$ contains an infinite dimensional
vector subspace $W$ of $V$ (here the topological structure of $V$ is not required);
moreover, if $\on{dim}W=\kappa$ then $A$ is called {\em \mbox{$\kappa$-lineable}};
\item {\em spaceable} if $A\cup\{ 0\}$ contains an infinite dimensional
closed vector subspace $W$ of $V$;
moreover, if $W$ is nonseparable, we say that $A$ is {\em nonseparably spaceable}.
\end{itemize}

One aim of our paper is to reexamine the spaceability of some families of singular functions contained in the Banach algebra $\on{CBV}$
of all continuous functions from $[0,1]$ to $\BB R$ of bounded variation, endowed with the norm
$$||f||=|f(0)|+\Var(f)$$
where $\Var (f)=\Var _{[0,1]} f$ denotes the total variation of $f$ in $[0,1]$ and, in general, $\Var _{[a,b]}f$ denotes the variation
of $f$ in a subinterval $[a,b]$ of $[0,1]$.
The lineability and spaceability of certain subfamilies of $\on{CBV}$ were studied in \cite{BG1} and more recently, in \cite{BG2}.
Our main result going in this direction states that the set of strongly singular functions is nonseparably
spaceable (Theorem \ref{t1}). We heavily exploit a family of such functions known in the probability theory \cite{Bi}.

Another aim of our paper is to establish strong $\ce$-algebrability of some families of singular functions.
Here $\ce$ denotes the cardinality of $\BB R$ (continuum). Algebrability and strong algebrability are associated with algebras, the structures
richer than linear spaces.
The notion of strong algebrability for various special subfamilies of $\on{CBV}$, $C[0,1]$ and  $\BB R^{[0,1]}$
becomes interesting in light of some recent results obtained in this context. We propose a new general criterion of strong
$\ce$-algebrability (Proposition \ref{p1}) which is used in Theorems \ref{t2} and \ref{t4}.

Recall (see \cite{ASS}) that a subset $E$ of an algebra $\mc A$ is {\em algebrable} if $E\cup\{ 0\}$ contains an
infinitely generated subalgebra $\mc B$ of $\mc A$.
If $E$ is algebrable with the minimal set of generators of $\mc B$ of
cardinality $\kappa$, then $E$ is called $\kappa$-algebrable.

A strengthened notion of algebrability was introduced in \cite{BG}. Given an infinite cardinal $\kappa$ and a commutative algebra $\mc A$,
a subset $E$ of $\mc A$ is called
{\em strongly $\kappa$-algebrable} whenever there exists a set $X=\{ x_\alpha\colon \alpha<\kappa\}\subset E$ of free generators of
a subalgebra $\mc B\subset E\cup\{0\}$ (that is, the set $\widehat{X}$ of all elements of the form
$x^{k_1}_{\alpha_1}x_{\alpha_2}^{k_2}\dots x_{\alpha_n}^{k_n}$, with nonnegative integers $k_1,\dots k_n$ nonequal to $0$ simultaneously,
is linearly independent and all linear combinations of elements from $\widehat{X}$
are in $E\cup\{0\}$). A set $E\subset\mc A$ is called {\em strongly algebrable} if it is strongly $\kappa$-algebrable for an infinite $\kappa$,
and it is {\em densely strongly $\kappa$-algebrable} if it is strongly $\kappa$-algebrable and the respective free subalgebra is dense in $\mc A$,
provided that $\mc A$ is a Banach algebra.

Let $\lambda$ stand for Lebesgue measure in $\BB R$.
A continuous function $f\colon[a,b]\to\BB R$ of bounded variation is said to be {\em singular} whenever it is not constant
and $f'=0$, $\lambda$-almost everywhere. The classical Cantor function
(see e.g. \cite{DMRV}) is an example of nondecreasing singular function defined in $[0,1]$. Also strictly increasing
singular functions are known, see \cite{T} where a good bibliography on this topic is presented.
We will consider classes of singular functions inside which some rich algebraic structures can be inscribed.

Note that $\on{CBV}$ is a subspace of the Banach algebra $\on{BV}$ of real-valued functions of bounded variation on $[0,1]$,
endowed with the same norm.
It is known that $\on{BV}$ can be treated as the space of finite signed Borel measures on [0,1]
with the norm being the total variation of such a measure.
Then $\on{CBV}$ is associated with those Borel measures which vanish on singletons. Among such measures, there is an important class of Borel
probability measures described in full by continuous distribution functions.
In particular, we will be interested in the following class of probability measures $\{\mu_p\colon p\in(0,1/2)\}$.
Namely, $\mu_p$ is the distribution of the sum $X=\sum_{k=1}^\infty(1/2^k)X_k$ where $X_k$, $k\in\BB N$, is a sequence
of independent random variables with $\on{Pr}(X_k=0)=p$ and $\on{Pr}(X_k=1)=1-p$.
The distribution function $F_p(t):=\on{Pr}(X\le t)$, $t\in\BB R$, associated with $\mu_p$ has the following properties (see \cite[\S 31]{Bi}):
\begin{itemize}
\item[(i)] Consider a birational number $t=k/2^n$, for $n\in\BB N$ and $k\in\{ 0,\dots ,2^n-1\}$, and the terminating binary expansion
$t=\sum_{k=1}^n u_k/2^k$ with $u_k\in\{ 0,1\}$. Let $\ell(t)$ and $r(t)$ denote (respectively) the numbers of zeros and ones among $u_1,\dots ,u_n$.
Then
$$\mu_p\left(\left[t,t+\frac{1}{2^n}\right]\right)=F_p\left(t+\frac{1}{2^n}\right)-F_p(t)=p^{\ell(t)}(1-p)^{r(t)}.$$
It follows that
$$\mu_p\left(\left[t,t+\frac{1}{2^n}+\frac{1}{2^{n+1}}\right]\right)=p^{\ell(t)+1}(1-p)^{r(t)}$$
$$\mu_p\left(\left[t+\frac{1}{2^n}-\frac{1}{2^{n+1}},t+\frac{1}{2^n}\right]\right)=p^{\ell(t)}(1-p)^{r(t)+1}.$$
\item[(ii)] $F_p$ is continuous and strictly increasing on $[0,1]$, $F_p(x)=0$ for all $x\le 0$ and $F_p(x)=1$ for all $x\ge 1$.
\item[(iii)] If $x\in (0,1)$ and $F'_p(x)$ exists then $F'_p(x)=0$.
As $F_p$ is monotone, $F'_p$ exists $\lambda$-almost everywhere in $[0,1]$
and consequently, $F'_p=0$, $\lambda$-almost everywhere in $[0,1]$.
\item[(iv)]  For any distinct $p,q\in (0,1/2)$ there are disjoint Borel sets $B_p,B_q\subset [0,1]$ such that
$\mu_p(B_p)=1$, $\mu_q(B_q)=1$. In other words, the measures $\mu_p$ and $\mu_q$ have disjoint supports.
For the proof, see \cite[Example 31.3]{Bi}.
\end{itemize}
In the sequel, we will consider the functions $F_p$, $p\in (0,1/2)$, restricted to $[0,1]$. The following result belongs to mathematical folklore.
The analogue of its second part is valid for any two Borel probability measures on $[0,1]$ having disjoint supports.
\begin{fact}  \label{F1}
The space $\on{CBV}$ is nonseparable. This is witnessed by the condition $||F_p-F_q||=2$ for any distinct $p,q\in (0,1/2)$.
\end{fact}
\begin{proof}
Consider any distinct $p,q\in (0,1/2)$. Pick $B_p$ and $B_q$ as in (iv). Clearly, $||F_p-F_q||=\Var_{[0,1]}(F_p-F_q)\le 2$.
Let $\ve\in(0,1/4)$.
Pick closed sets $C_p\subset B_p$ and $C_q\subset B_q$ such that $\mu_p(C_p)\ge 1-\ve$ and $\mu_q(C_q)\ge 1-\ve$.
Choose disjoint sets $G_p\supset C_p$ and $G_q\supset C_q$ that are open in $[0,1]$.
Let $(I_n)$ and $(J_n)$ stand for the sequences of all connected components of $G_p$ and $G_q$, respectively. Then
$\mu_p(G_p)\ge 1-\ve$, that is $\sum_n \Var_{\on{cl}(I_n)}F_p\ge 1-\ve$, and $\mu_p(G_q)\le\ve$, that is $\sum_n \Var_{\on{cl}(J_n)}F_p\le\ve$.
The analogous inequalities hold for measure $\mu_q$.
So,
$$1-\ve\le\sum_n \Var_{\on{cl}(I_n)}F_p=\sum_n \Var_{\on{cl}(I_n)}((F_p-F_q)+F_q)$$
$$\le \sum_n \Var_{\on{cl}(I_n)}(F_p-F_q)+\sum_n \Var_{\on{cl}(I_n)}F_q
\le \sum_n \Var_{\on{cl}(I_n)}(F_p-F_q) +\ve .$$
Consequently,
$$\sum_n \Var_{\on{cl}(I_n)}
(F_p-F_q)\ge 1-2\ve\text{ and analogously, }\sum_n \Var_{\on{cl}(J_n)}(F_p-F_q)\ge1-2\ve .$$
Hence we have
$$||F_p-F_q||=\Var_{[0,1]}(F_p-F_q)\ge\sum_n \Var_{\on{cl}(I_n)}(F_p-F_q)+\sum_n \Var_{\on{cl}(J_n)}(F_p-F_q)\ge 2-4\ve .$$
Letting $\ve\to 0$ we obtain the assertion.
\end{proof}

\section{Nonseparable spaceability of the set of strongly singular functions}

A singular function $f\in\on{CBV}$ will be called {\em strongly singular} whenever its restriction to every subinterval of $[0,1]$
is singular. In other words, $f\in\on{CBV}$ is strongly singular whenever $f'=0$ almost everywhere and $f$ is not constant in every interval.
Every function $F_p$, $p\in(0,1/2)$, considered in Section 1 is strongly singular.

For each strongly singular function $f$ and any subinterval $[a,b]\subset [0,1]$, $a<b$, we
have the following properties:
\begin{itemize}
\item[(I)] $f$ is nondifferentiable somewhere in $[a,b]$;
\item[(II)] $f$ is not absolutely continuous in $[a,b]$ since otherwise, $0=f(d)-f(c)=\int_c^d f'$ for all $c,d\in [a,b]$, $c<d$, (cf. \cite[Thm 4.14]{Go});
\item[(III)] $f$ is not Lipschitz in $[a,b]$ since every Lipschitz function is absolutely continuous in a given interval.
\end{itemize}

A main result of this section is the following.

\begin{thm} \label{t1}
The set of strongly singular functions in $\on{CBV}$ is nonseparably spaceable.
\end{thm}

An interesting example of nonseparably spaceable subset of $\on{BV}$ has been established recently in \cite[Thm 3.1]{GKP}.
Our result is an essential strengthening of \cite[Thm 4.1]{BG2} where it was proved that the set of nonabsolutely continuous
functions in $\on{CBV}$ is spaceable.

The proof of Theorem \ref{t1} will be divided into several lemmas.
Let $W=\on{span}\{ F_p\colon p\in(0,1/2)\}$, that is, $W$ denotes the linear subspace of $\on{CBV}$ generated
by all functions $F_p$, $p\in (0,1/2)$.

\begin{lem} \label{L1}
Each function from $W$ is not constant in every subinterval $I$ of $[0,1]$.
\end{lem}
\begin{proof}
By induction with respect to $k\in\BB N$, we will show a more general condition stating that
for any $0<p_1<\dots <p_k<1/2$ and $a_i\neq 0$ with $i=1,\dots ,k$, and for every interval
$I$  of the form $I=[(j/2^n,(j+1)/2^n]$ with $n\in\BB N$, $j=0,\dots ,2^n-1$, we have
$$\sum_{i=1}^k a_i\mu_{p_i}(I)\neq 0.$$
Suppose that we have proved this statement and suppose that there exists $f\in W$ which is constant
in some subinterval of $[0,1]$. Then $f$ is of the form
$$
f=\sum_{i=1}^k a_iF_{p_i}\text{ where }0<p_1<\dots <p_k<\frac{1}{2}\text{ and }a_i\neq 0\text{ with }i=1,\dots ,k.
$$
We may assume that $f$ is constant in an interval $I$ of the form as above. Then we get $\sum_{i=1}^k a_i\mu_{p_i}(I)=0$
which is impossible.

To start the induction, observe that our statement
for $k=1$ is obvious. Assume that this is true for a number $k\in\BB N$. Consider
$\sum_{i=1}^{k+1} a_i\mu_{p_i}$ where $0<p_1<\dots <p_{k+1}<{1}/{2}$ and $a_i\neq 0$ with $i=1,\dots ,k+1$.
Fix an interval $I=[j/2^n,(j+1)/2^n]$ with $n\in\BB N$ and $j=0,\dots ,2^n-1$.
Suppose that
$\sum_{i=1}^{k+1}a_i\mu_{p_i}(I)=0$.
Consider $J=[(2j)/2^{n+1},(2j+1)/2^{n+1}]$, the left half of $I$. We then have $\mu_{p_i}(J)=p_i\mu_{p_i}(I)$
for $i=1,\dots ,k+1$ (see property (i)). Using this, we obtain
$$\sum_{i=1}^{k+1}a_i\mu_{p_i}(J)=\sum_{i=1}^{k+1}a_ip_i\mu_{p_i}(I)
=p_{k+1}\sum_{i=1}^{k+1}a_i\mu_{p_i}(I)+\sum_{i=1}^{k}a_i(p_i-p_{k+1})\mu_{p_i}(I)$$
$$=\sum_{i=1}^{k}a_i(p_i-p_{k+1})\mu_{p_i}(I).
$$
Since $\sum_{i=1}^{k}a_i(p_i-p_{k+1})\mu_{p_i}(I)\neq 0$  by the induction hypothesis, we obtain
a contradiction.
\end{proof}
\begin{lem} \label{L2}
If $f$ belongs to the closure $\on{cl}(W)$ of $W$ in $\on{CBV}$ then $f'=0$ almost everywhere in $[0,1]$.
\end{lem}
\begin{proof}
Assume that $f_n\to f$ in the norm of $\on{CBV}$, for some sequence $(f_n)$ of functions from $W$.
For each $n\in\BB N$, let $D_n=\{x\in [0,1]\colon f'_n(x)=0\}$. Then $\lambda(D_n)=1$.
Since $f\in\on{CBV}$, the derivative
$f'$ exists almost everywhere in $[0,1]$. Suppose that $|f'|>0$ in a set $E$ of positive Lebesgue measure.
Then there exists $k\in\BB N$ such that $\lambda(E_k)=\alpha>0$ where $E_k=\{ x\in E\colon |f'(x)|>1/k\}$.
It follows that $|(f-f_n)'(x)|>1/k$ for all $x\in E_k\cap D_n$ and $n\in\BB N$.
Now, we have (for the first inequality, see \cite[Thm 224 I]{Fr})
$$\Var _{[0,1]}(f-f_n)\ge\int_{[0,1]}|(f-f_n)'|\ge\int_{E_k\cap D_n}|(f-f_n)'|\ge\frac{\alpha}{k}$$
for all $n\in\BB N$, which contradicts $||f_n-f||\to 0$.
\end{proof}
\begin{lem} \label{L3}
Consider arbitrary birational numbers $t_0={i_0}/{2^{n_0}}$ and $t_1={i_1}/{2^{n_1}}$ from $[0,1)$ such that
$n_1\ge n_0$, $\ell(t_1)\ge\ell(t_0)$ and $r(t_1)\ge r(t_0)$ (see property {\em (i)}). Put $I_0=[t_0,t_0+1/2^{n_0}]$ and
$I_1=[t_1,t_1+1/2^{n_1}]$. Then there exists a subinterval $J=[j/2^{n_1},(j+1)/2^{n_1+1}]$ (with $j\in\BB N$) of $I_0$ such that
$\mu_p(J)=\mu_p(I_1)$ for each $p\in(0,1/2)$. Moreover, for any real numbers $\alpha ,\beta\in I_1$,
$\alpha <\beta$, there exists a subinterval $[\alpha_1,\beta_1]$ of $I_0$ such that
$\mu_p([\alpha_1,\beta_1])=\mu_p([\alpha,\beta])$ for each $p\in(0,1/2)$.
\end{lem}
\begin{proof} Note that $n_0=\ell(t_0)+r(t_0)$ and $n_1=\ell(t_1)+r(t_1)$.
Let $m=\ell(t_1)-\ell(t_0)$ and $k=r(t_1)-r(t_0)$.
We make $m+k$ divisions into halves of consecutive intervals. We start from $I_0$, choosing left halves $m$ times
and then choosing right halves $k$ times. After that we obtain the interval
$$J=\left[t_0+\frac{1}{2^{n_0+m}}-\frac{1}{2^{n_0+m+k}},t_0+\frac{1}{2^{n_0+m}}\right].$$
Then using several times the final part of property (i), we have
$$\mu_p(J)=p^{\ell(t_0)+m}(1-p)^{r(t_0)+k}=p^{\ell(t_1)}(1-p)^{r(t_1)}=\mu_p(I_1)$$
for each $p\in(0,1/2)$, as desired.

To prove the second assertion, let $t_2=\min J$ and note that $\ell(t_2)=\ell(t_1)$, $r(t_2)=r(t_1)$.
Let $x=t_2-t_1$ and consider a subinterval $I$ of $I_1$ of the form
\begin{equation} \label{E3}
\left[\frac{i}{2^n},\frac{i+1}{2^n}\right]\text{ for }n\ge n_1\text{ and }i\in\{0,\dots ,2^n-1\}.
\end{equation}
Denote $I+x=\{t+x\colon t\in I\}$. Then $I+x\subset J$. Observe that $\min I$ and $\min(I+x)$ have the same numbers of zeros and ones
in their terminating binary expansions. So by property (i) we have
\begin{equation} \label{E4}
\mu_p(I+x)=\mu_p(I) \text{ for each } p\in\left(0,\frac{1}{2}\right) .
\end{equation}
Now, let $\alpha ,\beta\in I_1$,
$\alpha <\beta$. Then $[\alpha ,\beta]$ can be expressed as a countable union of
intervals of the form (\ref{E3}) with pairwise disjoint interiors. Since every measure $\mu_p$ vanishes on singletons,
from (\ref{E4}) it follows that $\mu_p([\alpha,\beta]+x)=\mu_p([\alpha,\beta])$ for each $p\in(0,1/2)$.
Thus we put $\alpha_1=\alpha+x$, $\beta_1=\beta+x$ and we obtain the assertion.
\end{proof}

\begin{lem} \label{L4}
If $f\in\on{cl}(W)$ is constant in some subinterval of $[0,1]$ then $f$
is equal to $0$ in $[0,1]$.
\end{lem}
\begin{proof}
Let $f\in\on{cl}(W)$ be constant in some subinterval of $[0,1]$. If $f$ is constant in $[0,1]$ then
since $f(0)=0$, we have $f=0$ in $[0,1]$ which yields the assertion. So, suppose that $f$ is not constant in $[0,1]$.
Pick a subinterval $I_0=[i_0/2^{n_0}, (i_0+1)/2^{n_0}]$ of $[0,1]$ such that $f|_{I_0}$ is constant and $I_0$ is the longest
among such subintervals of $[0,1]$. Let $f|_{I_0}=c$.

Since $f$ is not constant in $[0,1]$, we can choose $I=[i/2^n,(i+1)/2^n]$ such
that $I\subset [0,1]$, $n\ge n_0$ and $f|_I$ is not constant. We are going to find a subinterval $I_1=[i_1/2^{n_1},(i_1+1)/2^{n_1}]$
of $[0,1]$ such that $n_1\ge n_0$, $f|_{I_1}$ is not constant, and additionally
\begin{equation}\label{E5}
\ell(i_1/2^{n_1})\ge\ell(i_0/2^{n_0}),\quad r(i_1/2^{n_1})\ge r(i_0/2^{n_0}).
\end{equation}
If the interval $I$ staisfies (\ref{E5}), put $I_1=I$. Otherwise, observe that $n\ge n_0$
implies that for $i_1=i$ and $n_1=n$ at most one of
inequalities in (\ref{E5}) can fail. Assume, for instance, that $\ell(i/2^n)<\ell(i_0/2^{n_0})$.
(The case $r(i/2^n)<r(i_0/2^{n_0})$ is similar.) We will find a subinterval
$J$ of $I$ which can be taken as $I_1$. Namely, since $f|_I$ is not constant, pick $a,b\in I$, $a<b$, such that $f(a)\neq f(b)$.
The set $A=(f|_{[a,b]})^{-1}[\{f(a)\}]$ is compact and define $z=\max A$. Then $f$ cannot be constant in any interval
$[z,z+\delta]$ with $z+\delta<b$ and $\delta>0$. If $z$ is birational,
pick $m\in\BB N$ such that $z+1/2^m<b$ and put $J=[z,z+1/2^m]$. If $z$ is not birational, consider its binary expansion $z=\sum_j z_j/2^j$ and
pick $k\in\BB N$ so large that for
$z'=\sum_{j=1}^k z_j/2^j$ we have $a\le z'$ and $z'+1/2^{k+1}<b$. Then put $J=[z',z'+1/2^{k+1}]$.

Having $I_1$ defined, pick $x,y\in I_1$, $x<y$, such that $f(x)\neq f(y)$ and let $\ve=|f(x)-f(y)|$.
Since $f\in\on{cl}(W)$, choose $g\in W$ such that $||g-f||<\ve/4$. In particular
$$|g(x)-f(x)|<\frac{\ve}{4},\; |g(y)-f(y)|<\frac{\ve}{4} \text{ and } |g(t)-c|<\frac{\ve}{4}\text{ for all }t\in I_0.$$
It follows that $|g(y)-g(x)|>\ve/2$. By Lemma \ref{L3} pick an interval $[\alpha,\beta]\subset I_0$ such that
\begin{equation} \label{E6}
\mu_p([\alpha ,\beta])=\mu_p([x,y]) \text{ for each }p\in\left(0,\frac{1}{2}\right).
\end{equation}
Since $g\in W$, we can write
$$g=\sum_{i=1}^k s_iF_{p_i} \text{ where }0<p_1<p_2<\dots <p_k<\frac{1}{2} \text{ with }s_i\neq 0\text{ for }i=1,\dots ,k.$$
Let $g^+=\sum_{s_i>0}s_iF_{p_i}$ and $g^-=\sum_{s_i<0}s_iF_{p_i}$. Then $g=g^+-g^-$ and from (\ref{E6}) it follows that
$$g^+(\beta)-g^+(\alpha)=g^+(y)-g^+(x) \text{ and }g^-(\beta)-g^-(\alpha)=g^-(y)-g^-(x)$$
which implies that $g(\beta)-g(\alpha)=g(y)-g(x)$. Hence
$$\frac{\ve}{2}<|g(y)-g(x)|=|g(\beta)-g(\alpha)|\le|g(\beta)-c|+|c-g(\alpha)|<\frac{\ve}{2}$$
which yields a contradiction.
\end{proof}

{\bf Proof of Theorem \ref{t1}.} From Lemma \ref{L1} it follows that the functions $F_p$, $p\in(0,1/2)$, are linearly independent.
We have defined $W=\on{span}\{F_p\colon p\in(0,1/2)\}$. Observe that $\on{cl}(W)$
is a closed vector subspace of $\on{CBV}$ which is additionally nonseparable by Fact \ref{F1}.
By Lemmas~\ref{L2} and~\ref{L4}, the set $\on{cl}(W)\setminus\{ 0\}$ consists of
strongly singular functions. \hfill $\Box$

\section{Strong $\ce$-algebrability of some sets of singular functions}

We say that a function $f\colon \BB R\to\BB R$ is {\em exponential like (of range $m$)} whenever
\begin{equation} \label{E7}
f(x)=\sum_{i=1}^m a_i e^{\beta _i x},\quad x\in\BB R,
\end{equation}
for some distinct nonzero real numbers $\beta_1,\dots \beta_m$ and some nonzero real numbers $a_1,\dots a_m$.
We will also consider exponential like functions (of the same form) with domain $[0,1]$.

Our general criterion of strong $\ce$-algebrability is the following.
\begin{prop}  \label{p1}
Given a family $\mc F\subset{\BB R}^{[0,1]}$, assume that there exists a function $F\in\mc F$ such that
$f\circ F\in\mc F\setminus\{ 0\}$ for every exponential like
function $f\colon\BB R\to\BB R$. Then $\mc F$ is strongly $\ce$-algebrable. More exactly,
if $H\subset\BB R$ is a set of cardinalty $\ce$, linearly independent over the rationals $\BB Q$,
then $\exp\circ\,(rF)$, $r\in H$, are free generators of an algebra contained in $\mc F\cup\{ 0\}$.
\end{prop}
\begin{proof}
Fix a set $H$ of cardinalty $\ce$, linearly independent over the rationals $\BB Q$.  By the assumption,
$\exp\circ\,(rF)\in\mc F$ for all $r\in H$. To show
that $\exp\circ\,(rF)$, $r\in H$, are free generators of an algebra contained in $\mc F\cup\{ 0\}$,
consider any $n\in\BB N$ and a non-zero polynomial $P$ in $n$ variables without a constant term.
Then the function given by
\begin{equation} \label{Eq1}
x\mapsto P(e^{r_1 F(x)},e^{r_2 F(x)},\dots ,e^{r_n F(x)}),\quad x\in[0,1],
\end{equation}
is of the form
\begin{equation} \label{Eq2}
\sum_{i=1}^m a_i \left(e^{r_1 F(x)}\right)^{k_{i1}} \left(e^{r_2 F(x)}\right)^{k_{i2}}\dots \left(e^{r_n F(x)}\right)^{k_{in}}
=\sum_{i=1}^m a_i \exp\left(F(x)\sum_{j=1}^n r_j k_{ij}\right)
\end{equation}
where $a_1,\dots a_m$ are nonzero real numbers and the matrix $[k_{ij}]_{i\le m,j\le n}$ has distinct nonzero rows, with $k_{ij}\in\{ 0,1,2,\dots\}$.
Since the function
$t\mapsto\sum_{i=1}^m a_i \exp(t\sum_{j=1}^n r_j k_{ij})$ is exponential like, from (\ref{Eq2}) and the the assumption it follows that the function
(\ref{Eq1}) is in $\mc F\setminus\{ 0\}$. (For technical details concerning the role of the set $H$, compare with \cite{BGP} where a similar technique was used.)
\end{proof}

Note that the functions of type $e^{\beta x}$ were used to show the lineability of various sets of functions, then the generators of
the respective linear subspace were of the form $F(x)e^{\beta x}$ with the respectively chosen function $F$ from the considered set.
(See e.g. \cite{BG2}, \cite{JMS}.) In Proposition~\ref{p1},
instead of multiplication, we use superposition of $e^{\beta x}$ with $F$, in aim to show the (strong) algebrability of the considered set.
This new idea will be used below in Theorems \ref{t2} and \ref{t4}.

\begin{lem} \label{L5}
For every positive integer $n$, any exponential like function $f\colon [0,1]\to \BB R$  of range $m$, and
each $c\in\BB R$, the preimage $f^{-1}[\{c\}]$ has at most $m$ elements. Consequently, $f$ is not constant in
every subinterval of $[0,1]$.
\end{lem}
\begin{proof}
We proceed by induction.
If $m=1$, the function $f$ is of the form $f(x)=a e^{\beta x}$, $x\in [0,1]$, with $a\neq 0$ and $\beta\neq 0$. So $f$ is stricly monotone
and the property is obvious.

Assume that the property holds for all exponential like functions of range $m$.
Let $f(x)=\sum_{i=1}^{m+1} a_i e^{\beta _i x}$, $x\in[0,1]$,
for some distinct nonzero real numbers $\beta_1,\dots \beta_{m+1}$ and some nonzero real numbers $a_1,\dots a_{m+1}$.
Consider the derivative
$$f'(x)=\sum_{i=1}^{m+1} \beta_i a_i e^{\beta _i x}= e^{\beta_1 x}\left(\beta_1 a_1+\sum_{i=2}^{m+1} \beta_i a_i e^{(\beta _i-\beta_1) x}\right),
\quad x\in[0,1].$$
Note that $\gamma_i=\beta_i-\beta_1$, for $i=2,\dots ,m+1$, are nonzero distinct real numbers. So, we may apply the induction hypothesis
to $g(x)=\sum_{i=2}^{m+1} \beta_i a_i e^{\gamma_i x}$, $x\in [0,1]$, and $c=-\beta_1 a_1$. This shows that $(f')^{-1}[\{ 0\}]$
has at most $m$ elements. Hence $f$ has at most $m+2$ local extrema on $[0,1]$ (we should take into account one-sided extrema at $0$ and $1$ where
maybe $f'$ does not vanish). This implies that for each $c\in\BB R$, the preimage $f^{-1}[\{c\}]$ has at most $m+1$ elements, as desired.
\end{proof}

\begin{thm} \label{t2}
The set of strongly singular functions in $\on{CBV}$ is strongly $\ce$-algebrable.
\end{thm}
\begin{proof}
Fix a strongly singular function $F\in\on{CBV}$. For instance, let $F$ be the distribution function $F_{1/4}$ considered in the previous sections.
It suffices to check that the assumption of Proposition~\ref{p1} is valid with $\mc F$ equal to the set of strongly singular functions.
Consider an exponent like function $f$ given by (\ref{E7}),
for some distinct nonzero real numbers $\beta_1,\dots \beta_m$ and some nonzero real numbers $a_1,\dots a_m$.
Since $F'=0$ almost everywhere in $[0,1]$, we have
$$(f\circ\,F)'(x)=F'(x)\sum_{i=1}^m a_i\beta_ie^{\beta_i F(x)}=0\text{ for almost all }x\in[0,1].$$
Suppose that $f\circ\,F$ is constant in some subinterval $[c,d]$ of $[0,1]$ with $c<d$.
Since $F^{-1}$ is a continuous increasing bijection from $[0,1]$ onto $[0,1]$, the function $f=(f\circ\, F)\circ\,F^{-1}$
is constant in the interval $[F(c),F(d)]$ which contradicts Lemma \ref{L5}.
\end{proof}

Note that $\ce$ is the largest among the cardinalities $\kappa$ which can yield $\kappa$-algebrability of strongly singular functions.
By property (III) of strongly singular functions, our Theorem \ref{t2} implies the recent result \cite[Thm 2.1]{JMS}
stating that the set of continuous functions on $[0,1]$, which are a.e. differentiable, with a.e. bounded derivative
and are not Lipschitz, is $\ce$-lineable.

By virtue of Lemma \ref{L2}, the set of strongly singular functions in $\on{CBV}$ cannot be densely algebrable.
On the other hand, the set of strongly singular functions can be considered a subset of the Banach algebra $C[0,1]$ of continuous
functions from $[0,1]$ to $\BB R$, with the supremum norm, and we have the following result.

\begin{thm} \label{t3}
The set of strongly singular functions is a densely strongly $\ce$-algebrable subset of $C[0,1]$.
\end{thm}
\begin{proof}
In the proof of the previous theorem, by the use of Proposition \ref{p1}, we have obtained a free algebra $\mc A$ contained
in the $\mc F\cup\{ 0\}$ where $\mc F$ is the set of strongly singular functions. According to Proposition \ref{p1}, $\exp\circ\, (rF)$, $r\in H$,
are free generators of $\mc A$. Now, we additionally assume that $H$ contains the terms of a sequence $(r_n)_{n\ge 1}$ convergent to $0$.
Thanks to this, the closure $\on{cl}(\mc A)$ of $\mc A$ in $C[0,1]$ contains all constant functions since if $c\in\BB R$ then the sequence
$(c\exp\circ\, (r_nF))_{n\ge 1}$ converges uniformly to $c$. Let $\mc A^*$ be the algebra generated by the constant functions and the functions from $\mc A$.
Then $\on{cl}(\mc A^*)=\on{cl}(\mc A)$. By the Stone-Weierstrass theorem we have $\on{cl}(\mc A^*)=C[0,1]$, so $\mc A$ is dense in $C[0,1]$,
as desired.
\end{proof}

In the proof of the next theorem, the following elementary lemma will be needed. Every continuous function defined
on an interval will be treated as $0$ times differentiable.

\begin{lem} \label{L6}
Given functions $f$ and $g$ from $(a,b)$ into $\BB R$, and an integer $n\ge 1$, assume that $f$ and $fg$ are $n$ times differentiable,
$g$ is $n-1$ times differentiable and $f$ does not vanish in $(a,b)$. Then $g$ is $n$ times differentiable.
\end{lem}
\begin{proof}
We use induction. The case if $n=1$ is clear since $g=(fg)/f$. Assume that the property is valid for a fixed $n\ge 1$.
Suppose that $f$ and $fg$ are $n+1$ times differentiable,
$g$ is $n$ times differentiable and $f$ does not vanish in $(a,b)$. Hence $(fg)'=f'g+g'f$  is $n$ times differentiable,
and so is $f'g$. This implies that $g'f$ is $n$ times differentiable. By the induction hypothesis, $g'$ is $n$ times
differentiable and the proof is finished.
\end{proof}

By $C_n$ we denote the algebra of functions from $[0,1]$ into $\BB R$ that have $n$ continuous derivatives in $[0,1]$.

\begin{thm} \label{t4}
Given $n\ge 1$, the set of functions of class $C_n$ that do not have derivative of order $n+1$ somewhere
in any open subinterval of $[0,1]$ is strongly $\ce$-algebrable.
\end{thm}
\begin{proof}
Consider $F=F_{1/4}$, the strongly singular function used before. We will use it to construct an increasing function $G\colon [0,1]\to\BB R$
of class $C_n$ which does not have derivative of order $n+1$ somewhere
in any open subinterval of $[0,1]$. For $n=1$, $G(x)=\int_0^x F$, $x\in [0,1]$, is good. Having the respective function $\widetilde{G}$ for
a number $n$, we define $G(x)=\int_0^x\widetilde{G}$, $x\in [0,1]$, which is good for $n+1$.

We will apply Proposition \ref{p1} with the set $\mc F$ of functions in class $C_n$ that do not have derivative of order $n+1$ somewhere
in any open subinterval of $[0,1]$. Namely, we will check that
$f\circ\, G\in\mc F\setminus\{ 0\}$ for every exponential like
function $f\colon\BB R\to\BB R$. Fix an exponential like
function $f$ given by (\ref{E7}). Note that the functions $f$ and $f'$ are infinitely differentiable
and, by Lemma \ref{L6}, they have finitely many zeros in any bounded interval. Clearly, $f\circ\, G\in C_n$.
Fix an interval $(a,b)\subset [0,1]$. We will show that $f\circ\, G$ does not have derivative of order $n+1$ somewhere
in $(a,b)$. We may assume that $f$ and $f'$ have no zeros in $(a,b)$. Suppose that $f\circ\, G$ is $n+1$ times differentiable
in $(a,b)$. Since
$$(f\circ\, G)'(x)=(f'\circ\, G)(x)\cdot G'(x),\quad x\in (a,b),$$
$f'\circ\, G$ is $n$ times differentiable in (a,b) and it does not vanish in $(a,b)$,
we infer by Lemma \ref{L6} that $G'$ is $n$ times differentiable in $(a,b)$.
This is a contradiction.
\end{proof}

\noindent
{\bf Acknowledgements.} The first author would like to thank Luis Bernal-Gonz\'alez for his useful comments \cite{BG3} on articles \cite{BG1}, \cite{BG2}
that have been inspiration for our research.
The work of the first and of the second authors has been supported by the (Polish) National Center of Science Grant No. N N201 414939 (2010-2013).


\begin{thebibliography}{abc}
\bibitem{AGS} R. Aron, V.I. Gurariy, J.B. Soane, {\it Lineability and spaceability of set of functions on $\BB R$}, Proc. Amer. Math. Soc. 133 (2005), 793--803.
\bibitem{ASS} R. Aron, J.B. Seoane-Sep\'ulveda, {\it Algebrability of the set of everywhere surjective functions on $\BB C$},
Bull. Belg. Math. Soc. Simon Stevin 14 (2007), 25--31.
\bibitem{BG} A. Bartoszewicz, Sz. G{\l}\c{a}b, {\it Strong algebrability of sets of sequences and functions}, Proc. Amer. Math. Soc. 141 (2013), 827--835.
\bibitem{BGP} A. Bartoszewicz, Sz. G{\l}\c{a}b, T. Poreda, {\it On algebrability of nonabsolutely convergent series}, Linear Algebra Appl. 435 (2011), 1025--1028.
\bibitem{BG1} L. Bernal-Gonz\'alez, {\it Dense-lineability in spaces of continuous functions}, Proc. Amer. Math. Soc. 136 (2008), 3163--3169.
\bibitem{BG2} L. Bernal-Gonz\'alez, M.O. Cabrera, {\it Spaceability of strict order integrability}, J. Math. Anal. Appl. 385 (2012), 303--309.
\bibitem{BG3} L. Bernal-Gonz\'alez, e-mail letter to M. Balcerzak (October 16, 2011).
\bibitem{Bi} P. Billingsley, {\it Probability and Measure}, Wiley and Sons, New York 1979.
\bibitem{DMRV} O. Dovgoshey, O. Mario, V. Ryazanov, M. Vuorinen, {\it The Cantor function}, Expo. Math. 24 (2006), 1--37.
\bibitem{Fr} D. H. Fremlin, {\it Measure Theory. Broad Foundations}, vol.2, Torres Fremlin, Colchester 2003.
\bibitem{GKP} S. G{\l}\c{a}b, P. Kaufman, L. Pellegrini, {\it Spaceability and algebrability of sets of nowhere integrable functions},
Proc. Amer. Math. Soc. (doi:10.1090/S0002-9939-2012-11574-6), to appear.
\bibitem{Go} R.A. Gordon, {\it The Integrals of Lebesgue, Denjoy, Perron, and Henstock}, Amer. Math. Soc., Providence 1994.
\bibitem{G1} V.I. Gurariy, {\it Subspaces and bases in spaces of continuous functions}, Dokl. Akad. Nauk SSSR 167 (1966), 971--973 (in Russian).
\bibitem{G2} V.I. Gurariy, {\it Linear spaces composed of nondifferentiable functions}, C.R. Acad. Bulgarie. Sci. 44 (1991), 13--16.
\bibitem{JMS} P. Jim\'enez-Rodr\'{\i}guez, G.A.Mu\~noz-Fern\'andez. J.B. Seoane-Sep\'ulveda,
{\it Non-Lipschitz functions with bounded gradient and related problems}, Linear Algebra Appl. 437 (2012), 1174--1181.
\bibitem{T} L. Takacs, {\it An increasing continuous singular function}, Amer. Math. Monthly 85 (1978), 35--37.
\end{thebibliography}
\end{document}